\documentclass[12pt]{amsart}
\usepackage{amsmath,amssymb,amsfonts,amsopn,amscd}
\usepackage[margin=1.6cm]{geometry}
\usepackage{parskip}
\usepackage{enumerate}
\usepackage{enumitem}

\usepackage{graphicx}
\usepackage{import}
\usepackage{pstricks}
\usepackage{amsfonts}
\usepackage{array}
\usepackage{amsthm}
\usepackage{thmtools}
\usepackage{amssymb}
\usepackage{comment}

\usepackage{hyperref}
\usepackage[capitalise]{cleveref}

\newtheorem{theorem}{Theorem}[section]
\newtheorem{corollary}{Corollary}[theorem]
\newtheorem{lemma}[theorem]{Lemma}
\newtheorem{proposition}[theorem]{Proposition}
\newtheorem{definition}{Definition}[section]
\newtheorem*{theorem A}{Theorem A}
\newtheorem*{corollary A}{Corollary A}
\newtheorem*{theorem B}{Theorem B}
\newtheorem*{theorem C}{Theorem C}
\newtheorem*{theorem D}{Theorem D}
\newtheorem*{theorem E}{Theorem E}
\newtheorem*{theorem F}{Theorem F}

\newtheorem{remark}[theorem]{Remark}

\DeclareMathOperator{\GL}{GL}

\DeclareMathOperator{\SL}{SL}

\newcommand{\iso}{\mathbin{\kern.15em\widetilde{\hphantom{\hspace{.6em}}}\kern-.98em\rightarrow\kern.05em}}
\newcommand{\longiso}{\mathbin{\kern.3em\widetilde{\hphantom{\hspace{1.1em}}}\kern-1.55em\longrightarrow\kern.1em}}

\newcommand{\C}{\ensuremath{\mathbb{C}}}

\newcommand{\mfp}{\mathfrak{p}}

\newcommand{\bfG}{\ensuremath{\mathbf{G}}}

\title[Representations of $\bfG(\mathcal{O}_r)$]{\textbf{The Defining Characteristic case of the REPRESENTATIONS OF $\mathrm{GL}_{n}$ and $\mathrm{SL}_{n}$ over principal ideal local rings}}

\author{Nariel Monteiro}
\address{Department of Mathematics
University of California
Santa Cruz, CA 95064 U.S.A}
\email{namontei@ucsc.edu}

\begin{document}

\begin{abstract}

Let $W_{r}(\mathbb{F}_{q})$ be the ring of Witt vectors of length $r$ with residue field $\mathbb{F}_{q}$ of characteristic $p$. In this paper, we study the defining characteristic case of the representations of $\mathrm{GL}_{n}$ and $\mathrm{SL}_{n}$ over the principal ideal local rings $W_{r}(\mathbb{F}_{q})$ and $\mathbb{F}_{q}[t]/t^{r}$. Let ${\mathbf{G}}$ be either $\mathrm{GL}_{n}$ or $\mathrm{SL}_{n}$ and $F$ a perfect field of characteristic $p$, we prove that for most $p$ the group algebras $F[{\mathbf{G}}(W_{r}(\mathbb{F}_{q}))]$ and $F[{\mathbf{G}}(\mathbb{F}_{q}[t]/t^{r})]$ are not stably equivalent of Morita type. Thus, the group algebras $F[{\mathbf{G}}(W_{r}(\mathbb{F}_{q}))]$ and $F[{\mathbf{G}}(\mathbb{F}_{q}[t]/t^{r})]$ are not isomorphic in the defining characteristic case.

\end{abstract}

\maketitle

\section{Introduction}
 
Let $\mathcal{O}$ be a discrete valuation ring with a unique maximal ideal $\mathfrak{p}$ and having finite residue field $\mathbb{F}_{q}$, the field with $q$ elements where $q$ is a power of a prime $p$. Let $\pi$ be a fixed generator of $\mfp$. We denote by $\mathcal{O}_r$ the reduction of $\mathcal{O}$ modulo $\mathfrak{p}^r$. Similarly, given $\mathcal{O}'$ a second discrete valuation ring with the same residue field $\mathbb{F}_{q}$, we can define $\mathcal{O}'_r$. Two good examples to keep in mind are $\mathbb{F}_{p}[t]/t^{r}$ and $\mathbb{Z}_p/p^r \simeq \mathbb{Z}/p^r\mathbb{Z}$. The complex representations of general linear groups over the rings $\mathcal{O}_r$, i.e. $\operatorname{GL}_n(\mathcal{O}_r)$, have been extensively studied; see \cite{stasinski2017representations} for an overview and some history of the subject. In particular, Onn \cite{OnnUri2008Roag} has conjectured that there is an isomorphism of group algebras $\mathbb{C}[\operatorname{GL}_n(\mathcal{O}_r)]\cong\mathbb{C}[\operatorname{GL}_n(\mathcal{O}'_r)]$. When $r=2$, this conjecture was proven by Singla \cite{SinglaPooja2010Orog}. Assuming $p$ is odd, Singla proved a generalization of the conjecture for $r=2$ when $\GL_n$ is replaced by $\mathrm{SL}_n$ with $p \nmid n$ or the classical groups $\mathrm{Sp}_n$, $\mathrm{O}_n$ and $\mathrm{U}_n$ \cite{SinglaPooja2012ORoC}. Later, Stasinski proved this for $\mathrm{SL}_n$ for all $p$ when $r=2$ \cite{StasinskiAlexander2019Romo}. In \cite{monteiro2021}, the author generalized the previous results when $r=2$ over $\mathbb{C}$ to results over a sufficiently large field $K$ of characteristic $l \neq p$. For example, it was shown that there exists an isomorphism of group algebras $K[\operatorname{GL}_n(\mathcal{O}_2)] \cong K[\operatorname{GL}_n(\mathcal{O}'_2)]$ over a sufficiently large field $K$ of characteristic $l$ when $l\neq p$. Furthermore,  Stasinski showed that Onn's conjecture holds for $\GL_2$ in \cite{Stasinski2009}. In the opposite direction, Hassain and Singla in \cite{HassainSingla2022} proved that the group algebras $\C [\mathrm{SL}_2(\mathbb{Z}/2^r\mathbb{Z})]$ and $\C[\mathrm{SL}_2(\mathbb{F}_{2}[t]/t^{r})]$ are not isomorphic for any even positive integer $r>2$. For each $r\geq3$, Hadas recently showed in \cite{hadas2024} that there exists an integer $M$ (depending on $\mathrm{GL}_n$ and $r$) such that the complex group algebras $\C [\mathrm{GL}_n(W_{r}(\mathbb{F}_{q}))]$ and $\C[\mathrm{GL}_n(\mathbb{F}_{q}[t]/t^{r})]$ are isomorphic as long as $p$ is larger than $M$.

Given the results mentioned above and Onn's conjecture, it is natural to ask whether such a group algebra isomorphism exists over any field. Thus, in this paper, we explore the following question: Let $\bfG$ be either $\GL_{n}$ or $\SL_{n}$, considered as group schemes over $\mathcal{O}_r$, are the group algebras $F[\bfG(\mathbb{F}_{q}[t]/t^{r})]$ and $ F[\bfG(W_{r}(\mathbb{F}_{q}))]$ isomorphic for a field $F$ of the same characteristic $p$ as $\mathbb{F}_{q}$? 

We prove, in \Cref{section:defining}, that:

\begin{theorem A}
\label{Theorem:A}
Let $\bfG$ be either $\GL_{n}$ or $\SL_{n}$ and $F$ be a perfect field of the same characteristic $p$ as $\mathbb{F}_{q}$. Given $p \geq n$, the group algebras $F[\bfG(W_{r}(\mathbb{F}_{q}))]$ and $F[\bfG(\mathbb{F}_{q}[t]/t^{r})]$ are not stably equivalent of Morita type when:
\begin{itemize}
    \item $r=2$ and $p\geq 2n$;
    \item $r=3$ and $p\geq 3$;
    \item $r\geq 4$ and $p\geq 2$.
\end{itemize}

\end{theorem A}

\begin{corollary A}
The group algebras $F[\bfG(W_{r}(\mathbb{F}_{q}))]$ and $F[\bfG(\mathbb{F}_{q}[t]/t^{r})]$ are not isomorphic under the conditions of Theorem A.
\end{corollary A}

Note whether the two groups $\bfG(W_{r}(\mathbb{F}_{q}))$ and $\bfG(\mathbb{F}_{q}[t]/t^{r})$ have isomorphic group algebras is trivially true whenever the groups are isomorphic. Therefore, we first investigate what conditions guarantee that those groups are not isomorphic. The non-isomorphism of these groups has been explored. For example, when $r=2$, this question has been explored in the work of Sah \cite{sah1974cohomology,sah1977cohomology}, Serre \cite{serre1997abelian}, and Stasinski \cite{StasinskiAlexander2019Romo}, who proved that, in general, these groups are not isomorphic. In \cref{Section:p}, we explore different conditions that guarantee that the groups $\bfG(W_{r}(\mathbb{F}_{q}))$ and $\bfG(\mathbb{F}_{q}[t]/t^{r})$ are not isomorphic for general $r$. To be more specific, we show that their $p$-Sylow subgroups have different $p$-exponents, in general. Later in \cref{section:defining}, we study how this condition influences the group algebras. 
 
\subsection*{Acknowledgements}
Parts of this paper were established while I was at the NSF ASCEND Writing Retreat, and I would like to express my gratitude for this opportunity. I also thank Robert Boltje, George McNinch, Sunrose Shrestha, and Alexander Stasinski for helpful discussions. This material is based on work supported by the National Science Foundation MPS-Ascend Postdoctoral Research Fellowship under Grant No. 2213166. 
 
%%%%%%%%%%%%%%%%
%\section{Notational preliminaries}
%\label{Section:2}

%%%%%%%%%%%%%%%% SECTION 2 %%%%%%%%%%%%%%%%%%%%%%%%%%%%%%%%%%%
%\newpage
\section{The $p$-exponent of $\bfG(\mathcal{O}_r)$}
\label{Section:p}

In this section, we study the $p$-exponent of $\bfG(\mathcal{O}_r)$. The exponent of a finite group $G$ is the least common multiple of the orders of its elements denoted as $\mathrm{exp}(G)$. We denote by $\mathrm{exp}_p(G)$, the $p$-exponent of $G$, for the $p$-part of the exponent of $G$. Note that the $\mathrm{exp}_p(G)$ is the same as the exponent of a Sylow $p$-subgroup of $G$. Thus, to compute $\mathrm{exp}_p(G)$, we focus our attention on the orders of $p$ -elements of $G$.

\begin{lemma}
\label{lemma:bound-p}
    The $\mathrm{exp}_p(\bfG(\mathcal{O}_r)) \leq p \cdot \mathrm{exp}_p(\bfG(\mathcal{O}_{r-1}))$. Thus, $\mathrm{exp}_p(\bfG(\mathcal{O}_r)) \leq p^{r-1} \cdot \mathrm{exp}_p(\bfG(\mathcal{O}_1))$. 
\end{lemma}

\begin{proof}
    Let $G_r=\bfG(\mathcal{O}_r)$ and $\mathrm{exp}_p(G_{r-1})=z$. Given $g$ a $p$-element of $G_r$ then $g^z$ is an element of $K_{r-1}$. Thus, $$g^{zp}=(I+\pi^{r-1}X)^p=I+p\pi^{r-1}X=1$$ for some matrix $X$ with entries over $\mathcal{O}_r$. Thus, $\mathrm{exp}_p(G_r) \leq p \cdot \mathrm{exp}_p(G_{r-1})$. The second statement follows by induction. 
\end{proof}

From now on, we denote $G_r:=\bfG(\mathbb{F}_{q}[t]/t^{r})$ and $G'_r:=\bfG(W_{r}(\mathbb{F}_{q}))$.

\begin{proposition}
\label{Proposition:p-exponentr}
  Assume $p\geq n$ then the $p$-exponent of $G'_r=\bfG(W_{r}(\mathbb{F}_{q}))$  is $p^r$ and $p$-exponent of $G_r=\bfG(\mathbb{F}_{q}[t]/t^{r})$ is less than equal to $p^{\left\lceil \log_p(r) \right\rceil +1}$. Thus, $\mathrm{exp}_p(\bfG(W_{r}(\mathbb{F}_{q})))> \mathrm{exp}_p(\bfG(\mathbb{F}_{q}[t]/t^{r}))$ when

  \begin{itemize}
  \item $r=3$ and $p\geq 3$;
  \item $r\geq 4$ and $p\geq 2$.
  \end{itemize}

\end{proposition}

\begin{proof}
    Let $E_{12}$ be the matrix with $1$ at the $(1,2)$-entry and zero everywhere else.  Given $g=I+E_{12} \in \bfG(W_{r}(\mathbb{F}_{q}))$  then $$g^{p^{r-1}}=(I+E_{12})^{p^{r-1}}= I +\binom{p^{r-1}}{1}E_{12}=I+p^{r-1}E_{12} \neq I,$$  since $(E_{12})^2=0$. By similar computations, we have $g^{p^r}=I$, so $g$ has order ${p^r}$. By \cref{lemma:bound-p}, $\mathrm{exp}_p(G'_r) \leq p^{r-1} \cdot \mathrm{exp}_p(G'_{1})=p^r$. Thus, assuming $p\geq n$ the $p$-exponent of $\bfG(W_{r}(\mathbb{F}_{q}))$ is $p^r$, as $\mathrm{exp}_p(G'_{1})=p$.  

    To compute the $p$-exponent of $G_r$, let $g$ be a p-element of $G_r$ and $z=\left\lceil \log_p(r) \right\rceil$. Note that $g^p \in K^1$, as $\mathrm{exp}_p(G_{1})=p$. As $p^z\geq r$ and $p=0$ in $\mathcal{O}_r$, we have

    $$g^{p\cdot p^z}=(I+\pi X)^{p^z}= I+ \pi^{p^z}X^{p^z}=I,$$
    for some matrix $X$ with entries over $\mathcal{O}_r$. So we conclude that $p^{\left\lceil \log_p(r) \right\rceil} \leq \mathrm{exp}_p(G_r) \leq p^{\left\lceil \log_p(r) \right\rceil +1}$. 

    To show the last statement, we need to show $p^r >p^{\left\lceil \log_p(r) \right\rceil +1}$ which is equivalent to $p^{r-2}\geq r$. For a fixed prime $p$, the function $f(r)=p^{r-2}-r$ is an increasing function for $r\geq 3$. Note that $f(3)\geq 0$ as long as $p\geq 3$. When $p=2$, $f(4)=0$. Thus, we conclude that if $r= 3$ and $p\geq 3$ or when $r\geq4$ and $p\geq2$ then $\mathrm{exp}_p(G'_r)> \mathrm{exp}_p(G_r)$.    
\end{proof}
%$p^{r-1}> p^{\left\lceil \log_p(r) \right\rceil}$
%Now go back to the case $r=2$:

\subsection{The $p$-exponent of the Groups $G_2$ and $G'_2$}

Note when $r=2$, \cref{Proposition:p-exponentr} shows that $p$-exponent of $G_2=\bfG(\mathbb{F}_{q}[t]/t^{2})$ is less than equal to $p^2$. Thus, we cannot conclude from \cref{Proposition:p-exponentr} that $\mathrm{exp}_p(G'_2)=p^2> \mathrm{exp}_p(G_2)$. In fact, see Remark 2.7 for an example when we have equality.  In this subsection, we will show that if $p\geq 2n$ then $\mathrm{exp}_p(G'_2)> \mathrm{exp}_p(G_2)$.

Note that any $p$-element $g$ of $\bfG(\mathcal{O}_2)$ maps to an element of a Sylow $p$-subgroup of $\bfG(\mathbb{F}_{q})$. Thus, without loss of generality, let $g=A(1+\pi X)$ where $A$ is an upper triangular matrix with $1$ along the diagonal and $X$ a matrix with entries over $\mathcal{O}_2$. In the following lemma, we compute the powers of such $g$. 
 
\begin{lemma}
\label{6proposition:order}
Let $g=A(I+\pi X)$ a matrix with entries over $\mathcal{O}_2$. For any integer $n$, we have that $$g^n=A^n+\pi\sum\limits_{i=0}^{n-1} A^{n-i}XA^i.$$
\end{lemma}
  
 \begin{proof}
The lemma follows by induction, since 

\begin{equation*}
\begin{aligned}
     g^{n+1}=gg^n& = (A+ \pi Ax)(A^n + \pi\sum\limits_{i=0}^{n-1}A^{n-i}XA^i)\\
      & =A^{n+1} + \pi\sum\limits_{i=0}^{n-1}A^{n+1-i}XA^i+AXA^n \\
      & =  A^{n+1} + \pi\sum\limits_{i=0}^{n}A^{n+1-i}XA^i
  \end{aligned}
\end{equation*}
 as $\pi^2=0$.
\end{proof}

Thus, $g^p=A^p+\pi\sum\limits_{i=0}^{p-1} A^{p-i}XA^i $ and define $B:=\sum\limits_{i=0}^{p-1} A^{p-i}XA^i$. We prove that when $p\geq 2n$ the order of $g$ only depends on $A$, by showing that $$B=\sum\limits_{i=0}^{p-1} A^{p-i}XA^i=0 \bmod{\pi}.$$ We first need the following lemma.

\begin{lemma}
\label{6lemma:Chu}
Let $p\geq 2n$ then the sum $\sum\limits_{i=0}^{p-1}\binom{p-i}{k} \binom{i}{l} = 0 \bmod p $ for all $k,l<n$. 
\end{lemma}

 \begin{proof}
We note the following form of the Chu–Vandermonde identity \text{\cite[Equation 6.16]{jaynes2003probability}},  $$\sum\limits_{m=0}^{s}\binom{m}{j} \binom{s-m}{r-j}=\binom{s+1}{r+1}$$ which holds for any integers $j$, $r$, and $s$ satisfying $0 \leq j \leq r \leq s$. By letting $p=s$, $k+l=r$ and $j=l$, we have  $$\sum\limits_{i=0}^{p}\binom{p-i}{k} \binom{i}{l} =\binom{p+1}{k+l+1}.$$ Since $k+l+1 \leq 2n-1 < p$, $$\sum\limits_{i=0}^{p}\binom{p-i}{k} \binom{i}{l} =\binom{p+1}{k+l+1}= 0 \bmod p.$$
\end{proof}

%\begin{remark}
Note that \cref{6lemma:Chu} is why we assume that $p\geq 2n$ when $r=2$. We use the previous lemma to prove the following. 
%\end{remark}

 \begin{lemma}
 \label{6proposition:B=0}
 Let $p\geq 2n$ and $A$ an upper triangular matrices of $\mathrm{M}_n(\mathcal{O}_2)$ with $1$'s along the diagonal. We have that $B:=\sum\limits_{i=0}^{p-1} A^{p-i}XA^i$ is equal to $0 \bmod \pi$.
  \end{lemma}

 \begin{proof}
Since $A$ is an upper triangular matrix with $ 1$'s along the diagonal, we can write $A=I+J$ where $I$ is the identity matrix and $J$ is a strictly upper triangular matrix. Since the matrix $I$ and $J$ commute, we have that $$A^m=(I+J)^m=\sum\limits_{k=0}^{m} \binom{m}{k} J^k.$$ Thus, $$B=\sum\limits_{i=0}^{p-1} A^{p-i}XA^i=\sum\limits_{i=0}^{p-1} (\sum\limits_{k=0}^{p-i} \binom{p-i}{k} J^k)X( \sum\limits_{l=0}^{i} \binom{i}{l})J^l)=\sum\limits_{i=0}^{p-1}\sum\limits_{k=0}^{p-i} \sum\limits_{l=0}^{i} \binom{p-i}{k} \binom{i}{l} J^k X J^l.$$ It is enough to consider $k,l<n$ as $J^n=0$ for a strictly upper triangular matrix. For a fixed $k$ and $l$, the coefficient of $J^k X J^l$ is $$\sum\limits_{i=0}^{p-1}\binom{p-i}{k} \binom{i}{l}= 0 \bmod p$$ by \cref{6lemma:Chu}, as $k,l<n$. Thus $B=0 \bmod \pi$ as the characteristic of the field is $p$.
\end{proof}

%\begin{proposition}
%The $p$-exponent of $\mathrm{GL}_{n}(\mathbb{F}_{q})$ for $p \geq n$ is $p$. 
%(proof or cite, generalize)
%\end{proposition} just say 

%\begin{proof}
%Note that the set of upper uni-triangular matrices forms a Sylow $p$-subgroup. Let $A = I + N$ denote any such matrix, with $N$ a nilpotent upper-triangular matrix. Note that $N^m = 0$ when $m \geq n$. Thus $A^p = I + N^p = I$ when $p \geq n$.
%\end{proof}

%Recall that for $p \geq n$ the $p$-exponent of $G_1=\mathrm{GL}_{n}(\mathbb{F}_{q})$  is $p$. 

We thus compute the $p$-exponent of the group $G_2$. 

\begin{proposition}
\label{Proposition:p-exponent}
  Let $p \geq 2n$, then the $p$-exponent of $\bfG(W_{2}(\mathbb{F}_{q}))$  is $p^2$ and that of $\bfG(\mathbb{F}_{q}[t]/t^{2})$ is $p$.   
\end{proposition}

\begin{proof}

Note that by \cref{Proposition:p-exponentr} the $p$-exponent of $G'_2$ is $p^2$. Let $g$ be a $p$-element of $G_2$. Without loss of generality, assume that $g=A(I+\pi X)$ where $A$ is an upper triangular matrix with 1's along the diagonal. 
By \cref{6proposition:order}, $$g^p=A^p+\pi\sum\limits_{i=0}^{p-1} A^{p-i}XA^i.$$ Since $p\geq 2n$,  by \cref{6proposition:B=0} we have that $B=\sum\limits_{i=0}^{p-1} A^{p-i}XA^i=0 \bmod \pi$  thus $g^p=A^p.$
 
Note that the extension  $$I \to K^1 \to G_2 \to G_1 \to I$$ split. Since for $p \geq n$ the $p$-exponent of $G_1=\bfG(\mathbb{F}_{q})$ is $p$, we can take $A$ to be a matrix of order $p$ by a choice of section from $G_1$. Therefore, $g$ also has order $p$ and the $p$-exponent of $G_2$ is $p$. 
\end{proof}
%However, for a $p$-element $g$ in $G'_2$ this is not necessarily true. For example, if  $g=A=I+E_{12}$ then $$g^p=A^p=(I+E_{12})^p= I +\binom{p}{1}E_{12}=I+pE_{12} \neq I .$$ Thus, $g$ does not have order ${p}$. For any element $g$ of a Syllow $p$-subgroup of $ G'_2$, $g^p$ is an element of the kernel, as $I=g^p \bmod \pi$. Moreover, any element of the kernel has order $p$, as $(I+\pi X)^p=I+ \pi (pX)=I$. Any $g$, a $p$-element in $G'_2$, has order at most $p^2$. Thus the $p$-exponent of $G'_2$ is $p^2$. 
\begin{remark}
    \cref{Proposition:p-exponent} shows that the groups $\bfG(\mathbb{F}_{q}[t]/t^{2})$ and $\bfG(W_{2}(\mathbb{F}_{q}))$ are not isomorphic when $p\geq 2n$, as the two groups have different $p$-exponent. Note that if $p < 2n$, then it is possible for the $p$-exponent of $\bfG(\mathbb{F}_{q}[t]/t^{2})$ to be $p^2$ and thus the same as the $p$-exponent of $\bfG(W_{2}(\mathbb{F}_{q}))$. For example, when $n=3$ in $\bfG(\mathbb{F}_{5}[t]/t^{2})$ the matrix $\begin{pmatrix}1 & 1 & 0\\
    t & 1 & 1\\
    t & 0 & 1\\
\end{pmatrix}$ has order $25$. Thus the $p$-exponent of $\bfG(\mathbb{F}_{5}[t]/t^{2})$ is $25$, the same as the the $p$-exponent of $\bfG(W_{2}(\mathbb{F}_{5}))$.

%It was also pointed out to us that we should read the Russian paper.
    
\end{remark}

%\newpage

\section{The Defining Characteristic Case of the Representations of $G_r$ and $G'_r$}
\label{section:defining}
%\subsection{Non-isomorphic Group Algebras of $G_r$ and $G'_r$ in the Defining Characteristic Case }

Throughout this section, $F$ denotes a perfect field of characteristic $p > 0$. We use the fact that $\bfG(\mathcal{O}_r)$ and $\bfG(\mathcal{O}'_r)$ have different $p$-exponent in general to show that over a perfect field $F$ of characteristic $p$ the two group algebras are not stably equivalent of Morita type. 

For an algebra $A$ over a field $F$ of characteristic $p$, we let $[A,A]$ denote the subspace spanned by commutators in $A$. We let  $$T_n(A):=\{x \in A: x^{p^n}\in [A,A]\} $$ and 
\[ 
T(A):= \bigcup_{n=0}^{\infty}T_n(A), 
\]
which are called the Külshammer spaces of $A$. We summarize some of the general properties of the Külshammer spaces in the following lemma. 
%subsection{Külshammer Spaces }

%\begin{definition}
%We first consider the \textbf{commutators} in $A$; these are all the elements of the form $ab-ba$ for $a$ and $b$ elements of $A$. We can consider the $F$-subspace $[A,A]$ spanned by all the commutators in $A$, called the \textbf{commutator subspace} of $A$.
%\end{definition}

%\begin{example}
%For example:
%\begin{enumerate}
    %\item When $A=\M_n(F)$  one can show that $[A,A]$ is the set of matrices with trace zero.
    %\item Let $A=FG$ for a finite group $G$ then $[A,A]=\{ \sum_{g \in G}a_{g}g \mid \sum_{g \in J}a_{g}=0 \ \text{for $J \in Cl(G)$} \}$ where  $Cl(G)$ denotes the set of conjugacy classes of $G$.
%\end{enumerate}
  
%\end{example}

% the interested reader can consult \cite{zbMATH00067401} for proofs of these facts. 

%\begin{definition}
%We can also define the \textbf{Külshammer spaces} $$T_n(A):=\{x \in A: x^{p^n}\in [A,A]\} $$ and 
%\[ 
%T(A):= \bigcup_{n=0}^{\infty}T_n(A). 
%\]

%\end{definition}

\begin{lemma} 
\cite[Proposition 2.9.5]{zimmermann2014representation}
Let $F$ be a field of characteristic $p > 0$ and $A$ be a finite-dimensional associative unitary $F$ algebra. Then $T_n(A)$ is a $F$-subspace of $A$ and
\[ 
[A, A] = T_0(A) \subseteq T_1(A) \subseteq T_2(A) \subseteq \dots \subseteq T(A)= Rad(A) + [A,A]
\] is an increasing sequence of $F$-subspaces of $A$.
\end{lemma}

Let $A$ be a symmetric $F$-algebra. Then there is a non-degenerate associative symmetric bilinear form $( \quad,\quad )_A : A \times A \to F.$  Given a subspace $V$ of $A$ we define $V^\perp=\{x \in A \, |\, (x,v)=0, \,\text{for all} \,v \,\text{in} \,V\}$. Let $R(A) := Z(A)\cap \mathrm{soc}(A)$ be the Reynolds ideal of $A$, where $Z(A)$ is the center of $A$ and $\mathrm{soc}(A)$ is the socle of $A$. Note that the group algebra $FG$ is a symmetric $F$-algebra where $(g,h)=\delta_{g,h^{-1}}$ for $g$, $h$ in $G$ and $\delta$ is the Kronecker symbol.

\begin{lemma} 
\cite[Proposition 2.9.5]{zimmermann2014representation}
Let $F$ be a field of characteristic $p > 0$ and $A$ be a finite-dimensional associative unitary $F$ algebra. Then $T_n(A)^{\perp}$ is an ideal of $Z(A)$ and
\[ 
R(A) = T(A)^ \perp = \bigcap_{n=0}^{\infty} T_n(A)^{\perp} \subseteq \dots \subseteq T_1(A)^{\perp}\subseteq T_0(A)^{\perp}= Z(A)
\] is an increasing sequence of ideals of $Z(A)$. 
\end{lemma}

The ideal $T_n(A)^{\perp}$ is called the $n$-th Külshammer ideal of $A$. For more properties on the Külshammer space and ideals, one can consult the original writing of Külshammer \cite{zbMATH00067401} or \cite[Section 2.9]{zimmermann2014representation}. For a group algebra $FG$ there exists an explicit description of the $n$-th Külshammer ideal of $FG$; $T_n(FG)^{\perp}$ has a basis given by 

$$ \bigg\{\sum_{g \in C^{p^{-n}}} g \, |\, C \in Cl(G)\, and \,C^{p^{-n}}\neq \emptyset \bigg\}
$$

where $Cl(G)$ denotes the set of conjugacy classes of $G$ and $C^{p^{-n}} := \{h \in G \,| \,h^{p^{n}} \in C\}$, see \cite[Lemma 2.9.20]{zimmermann2014representation}.

\begin{proposition}
\label{6proposition:pKul}
\cite[(20)]{zbMATH00067401} Given a finite group $G$ and $F$ a field of characteristic $p > 0$, the $p$-exponent of $G$ is equal to the quantity $min\{p^n:T_n(FG)^{\perp}=T(FG)^{\perp}=R(FG)\}$.  
%\cite[Lemma 2.9.18]{zimmermann2014representation}
\end{proposition}

\begin{proof}

 %Note that for an arbitrary $F$ subspace $V$ of a symmetric algebra $A$ we have $V^{\perp\perp}=V$, \cite[(30)]{zbMATH00067401}. Thus, $T_n(FG)= T(FG)$ if and only if $T_n(FG)^{\perp}= T(FG)^{\perp}$. Note that the $p$-exponent of $G$ is equal to the quantity $min\{p^n:T_n(FG)= T(FG)\}$ by \cite[(20)]{zbMATH00067401}. Therefore, the result follows as $T(FG)^{\perp}=R(FG)$. 

%[Write some notes about this]

%It was proven in \cite{zbMATH00067401}[No. (20)] that the the $p$-exponent of $G$ is equal to the quantity $min\{p^n:T_n(FG)= Rad(FG)+[FG,FG]\}$. The proof follows the same argument by replacing $T_n(FG)$ by $T_n(FG)^{\perp}$. We write it here for completion. 

Let $p^n$ be the $p$-exponent of $G$. Uniquely decomposing each $g$ in $G$ into commuting elements $g_p$ and $g_{p'}$ of order $p$ and $p'$ respectively, we have $g^{p^n} = g_p^{p^n}g_{p'}^{p^n}$. For $g$ and $h$ in $G$ we have that $g^{p^n}$ and $h^{p^n}$ are conjugates in $G$ if and only if  $g^{p^i}$ and $h^{p^i}$ are conjugates in $G$ for any $i\geq n $. Then by the explicit description of the Külshammer ideal of $FG$, we see that $T_i(FG)^{\perp} = T(FG)^{\perp}$ if and only if the order of $g^{p^i}$ is not divisible by $p$ for all $g$ in $G$. Therefore, the result follows as $T(FG)^{\perp}=R(FG)$.         
\end{proof}

\subsection{Equivalence of categories}\label{categories}

We first recall the definition of a stable equivalence of Morita type.

\begin{definition}
      \cite{broue1994equivalences} Let $F$ be a commutative ring and let $A$ and $B$ be $F$-algebras.
Let $M$ be a $B$–$A$-bimodule and let $N$ be an $A$–$B$-bimodule such that:
\begin{enumerate}
    \item 
    \begin{enumerate}
        \item $M$ is projective as a $B$-module and as an $A^{op}$-module.
        \item $N$ is projective as a $B^{op}$-module and as an $A$-module.
    \end{enumerate}

\item 
\begin{enumerate}
    \item $M \otimes_{A} N \simeq B \oplus Q$ as $B$–$B$-bimodules, where $Q$ is a finitely generated
projective $B$–$B$-bimodule.

\item $N \otimes_{B} M \simeq A \oplus P$ as $A$–$A$-bimodules, where $P$ is a finitely generated
projective $A$–$A$-bimodule.
\end{enumerate}

\end{enumerate}
Then we say that $(M,N)$ induces a stable equivalence of Morita type between $A$
and $B$.

\end{definition}

Let $A^e := A \otimes_k A^{op}$. The center of an algebra $A$ can be considered as the set of homomorphisms as $A^e$-modules from $A$ to $A$, that is, $Z(A) = Hom_{A^e} (A, A)$. The projective center $Z^{pr}(A) \subseteq Z(A)$ is defined as the subset of homomorphisms that factor through projective $A^e$-modules. The stable center $Z^{st}(A)$ is defined as the quotient $Z(A)/Z^{pr}(A)$. Note that by \cite[ Lemma 4.1 (iii)]{hethelyi2005central} and \cite[Proposition 1.3]{Projective}, $Z^{pr}(A)\subseteq R(A) \subseteq T_n(A)^{\perp}$ for a symmetric $F$-algebra $A$.

% see before Proposition 5.1. TRANSFER MAPS IN HOCHSCHILD (CO)HOMOLOGY AND APPLICATIONS TO STABLE AND DERIVED INVARIANTS AND TO THE AUSLANDER–REITEN CONJECTURE

We now prove Theorem A, which we restate here for the reader's convenience. 

\begin{theorem A}
\label{Theorem:main}
Let $\bfG$ be either $\GL_{n}$ or $\SL_{n}$ and $F$ be a perfect field of the same characteristic $p$ as $\mathbb{F}_{q}$. Given $p \geq n$, the group algebras $F[\bfG(W_{r}(\mathbb{F}_{q}))]$ and $F[\bfG(\mathbb{F}_{q}[t]/t^{r})]$ are not stably equivalent of Morita type when:
\begin{itemize}
    \item $r=2$ and $p\geq 2n$;
    \item $r=3$ and $p\geq 3$;
    \item $r\geq 4$ and $p\geq 2$.
\end{itemize}

\end{theorem A}

\begin{proof}

For $F$ a perfect field of characteristic $p > 0$, a stable equivalence of Morita type between two symmetric $F$-algebras $A$ and $B$ preserves the Külshammer ideals. Being more specific, for $n\geq 0$, we have that $T_n(B)^{\perp}/Z^{pr}(B)$ and $T_n(A)^{\perp} /Z^{pr}(A)$ are isomorphic as ideals via an algebra isomorphism $Z(B)/ Z^{pr}(B) \to Z(A)/ Z^{pr}(A)$, see \cite[Proposition 5.8]{konig2012transfer}. Assuming that $T_n(A)^{\perp}=T(A)^{\perp}$ for some $n$, we have $T_n(B)^{\perp}/ Z^{pr}(B)=T(B)^{\perp}/ Z^{pr}(B)$ under a stable equivalence of Morita type between the algebras $A$ and $B$. As $Z^{pr}(B)\subseteq T(B)^{\perp} \subseteq T_n(B)^{\perp}$, we have $T_n(B)^{\perp}=T(B)^{\perp}$.
   
%; see \cite[Proposition 5.9.21]{zimmermann2014representation}

%By \cite[Proposition 5.8]{konig2012transfer} the Külshammer ideals are stably invariant. Given a stable equivalent of Morita type between two symmetric algebras $A$ and $B$, we have $T_n(B)^{\perp}/Z^{pr}(B)$ and $T_n(A)^{\perp} /Z^{pr}(A)$ are isomorphic as ideals for all $n$ via an algebra isomorphism $Z(B)/ Z^{pr}(B) \to Z(A)/ Z^{pr}(A)$. Assume $T_n(A)^{\perp}=T(A)^{\perp}$ for some $n$ then we have that $T_n(B)^{\perp}/ Z^{pr}(B)=T(B)^{\perp}/ Z^{pr}(B)$. As $Z^{pr}(B)\subseteq R(B) \subseteq T_n(B)^{\perp}$ we have $T_n(B)^{\perp}=T(B)^{\perp}$. 
    
Thus, a stable equivalent of Morita type between the two group algebras $F[\bfG(W_{r}(\mathbb{F}_{q}))]$ and $F[\bfG(\mathbb{F}_{q}[t]/t^{r})]$ is only possible if the groups have the same $p$-exponent by \cref{6proposition:pKul}. By \cref{Proposition:p-exponentr} and \cref{Proposition:p-exponent}, the $p$-exponent of the two groups are different in the cases listed above. Thus, there is no stable equivalence of Morita type between the group algebras $F[\bfG(W_{r}(\mathbb{F}_{q}))]$ and $F[\bfG(\mathbb{F}_{q}[t]/t^{r})]$.
\end{proof}
%see cor 6.5 of Y. Liu state it here.

%Note that for any $A$ and $B$ finite dimensional k-algebras which are stably equivalent of Morita type, we have that 
%$$\dim(T_n(A)/K(A)) = \dim(T_n(B)/K(B)).$$ or can we use this

% may say that one is strictly contained in the other and contains K(A), so the levels cannot be isomorphic. 
\begin{corollary}
The group algebras $F[\bfG(W_{r}(\mathbb{F}_{q}))]$ and $F[\bfG(\mathbb{F}_{q}[t]/t^{r})]$ are not isomorphic under the conditions of Theorem A.
\end{corollary}

\begin{proof}
  There are stronger equivalences between $F$-algebras than a stable equivalence of Morita type. Note that for self-injective algebras we have:

$$
\begin{array}
[c]{cccccccc}

F- \text{algebra isomorphism} & \Longrightarrow &\text{Morita equivalence} \\
& & \Downarrow & & \\
\text{Stable equivalence of Morita type} & \Longleftarrow & \text{derived equivalence}&

\end{array}
$$
see \cite{broue1994equivalences}. 

%\cite{broue1994equivalences,zimmermann2019equivalences}

Note that group algebras over a field $F$ are self-injective algebras \cite[Theorem 4.11.2]{linckelmann2018block1}. Thus, it follows that none of the stronger equivalences in the diagram above is satisfied for the group algebras $F[\bfG(W_{r}(\mathbb{F}_{q}))]$ and $F[\bfG(\mathbb{F}_{q}[t]/t^{r})]$. In particular, the group algebras $F[\bfG(W_{r}(\mathbb{F}_{q}))]$ and $F[\bfG(\mathbb{F}_{q}[t]/t^{r})]$ are not isomorphic algebras.
\end{proof}

%It would be interesting to know if there is a weaker equivalence between the two group algebras $F[\bfG(W_{r}(\mathbb{F}_{q}))]$ and $F[\bfG(\mathbb{F}_{q}[t]/t^{r})]$ such as a perfect isometry introduced in \cite{Broue1990}.

\vspace{2mm}

\bibliographystyle{plain}
\bibliography{Bibliography}

\end{document}